\documentclass[a4paper,10pt]{amsart}

\usepackage{amscd}
\usepackage{xypic}  
\usepackage{amssymb}
\usepackage{amsthm}
\usepackage{epsfig}
\usepackage{paralist}

\usepackage{comment}

\usepackage[T1]{fontenc}
\usepackage[all,cmtip]{xy}
\usepackage{amsmath}
\usepackage{color}
\DeclareMathOperator{\coker}{coker}

\newtheorem{thm}{Theorem}[section]

\newtheorem{lemma}[thm]{Lemma}
\newtheorem{prop}[thm]{Proposition}

\newtheorem{definition}[thm]{Definition}

\newtheorem{corollary}[thm]{Corollary}
\theoremstyle{definition}
\newtheorem{ass}[thm]{Assumption}
\newtheorem{remark}[thm]{Remark}

  \newtheorem{definition-remark}[thm]{Definition-Remark}

\def\ker{\operatorname{ker}}
\def\coker{\operatorname{coker}}

\def\min{\operatorname{min}}
\def\im{\operatorname{im}}

\def\max{\operatorname{max}}

\def\c1{\operatorname{c_1}}
\def\c2{\operatorname{c_2}}

\def\CC{{\mathbb C}}

\def\PP{{\mathbb P}}
\def\A{{\mathcal A}}

\def\Rc{{\mathfrak R}}

\def\M{{\mathcal M}}
\def\N{{\mathcal N}}
\def\O{{\mathcal O}}
\def\I{{\mathcal J}}

\def\Z{{\mathcal Z}}

\def\T{{\mathcal T}}
\def\H{{\mathcal H}}
\def\F{{\mathcal F}}
\def\K{{\mathcal K}}

\def\V{{\mathcal V}}

\def\K{{\mathcal K}}

\def\ee{\mathfrak{e}}

\def\FF{\mathfrak{F}}

\def\cong{\simeq}

\def\sub{\subseteq}

\def\+{\oplus}                   
\def\*{\otimes}                  

\def\Shext{\operatorname{ \mathfrak{E}\mathfrak{x}\mathfrak{t} }}

\def\Pic{\operatorname{Pic}}

\def\Sing{\operatorname{Sing}}

\begin{document}

\title{Degeneration of differentials and moduli of nodal curves on $K3$ surfaces}

\author{C. Ciliberto}
\address{Ciro Ciliberto, Dipartimento di Matematica, Universit\`a di Roma Tor Vergata, Via della Ricerca Scientifica, 00173 Roma, Italy}
\email{cilibert@mat.uniroma2.it}

\author{F. Flamini}
\address{Flaminio Flamini, Dipartimento di Matematica, Universit\`a di Roma Tor Vergata, Via della Ricerca Scientifica, 00173 Roma, Italy}
\email{flamini@mat.uniroma2.it}

\author{C. Galati}
\address{Concettina Galati, Dipartimento di Matematica e Informatica, Universit\`a della Calabria, via P. Bucci, cubo 31B, 87036 Arcavacata di Rende (CS), Italy}
\email{galati@mat.unical.it}

\author{A. L. Knutsen}
\address{Andreas Leopold Knutsen, Department of Mathematics, University of Bergen, Postboks 7800,
5020 Bergen, Norway}
\email{andreas.knutsen@math.uib.no}




\begin{abstract}  
We consider, under suitable assumptions, the following situation: $\mathcal B$ is a component of the moduli space of polarized surfaces and $\mathcal V_{m,\delta}$ is the universal Severi variety over $\mathcal B$ parametrizing pairs $(S,C)$, with $(S,H)\in \mathcal B$ and $C\in |mH|$ irreducible with exactly $\delta$ nodes as singularities. The moduli map $\mathcal V\to \mathcal  M_g$ of an irreducible component $\mathcal V$ of  $\mathcal V_{m,\delta}$ is generically of maximal rank if and only if certain cohomology vanishings hold. Assuming there are suitable semistable degenerations of the surfaces in $\mathcal B$, we provide  sufficient conditions for the existence of an irreducible component $\mathcal V$ where these vanishings are verified.  As a test, we apply this to $K3$ surfaces and  give a new proof of a result recently independently proved by Kemeny and by the present authors. 

\end{abstract}

\maketitle



\medskip

%
%

\section{Introduction}\label{S:setup}

Let $(S, H)$  be a smooth, projective, polarized, complex  surface, with $H$ an ample line bundle such that the linear system $|H|$ contains smooth, irreducible curves.  We set  
\[
p:=p_a(H)=  \frac{1}{2} (H^2 + K_S\cdot  H) +1,
\]
 the arithmetic genus of any curve in  $|H|$. For any integer $m \geqslant 1$, we set 
\begin{equation*}\label{eq:Hp}
\ell(m):=\dim (|mH|)  \;\; \mbox{and} \;\; p(m) := p_a(mH)=  \frac{m(m-1)}{2} H^2 +  m (p -1).
\end{equation*} One has  
\[
\ell(m)= \chi(\O_S) +m^2 H^2 -p(m), \;\; \mbox{for } \;\; m\gg 0.
\]

For any integer $\delta\in \{0,\ldots,\ell (m)\}$, consider the locally closed, functorially defined subscheme of $|mH|$  
\begin{equation*}\label{eq:Sevvar2}
V_{m, \delta} (S,H) \; \mbox{(simply} \; V_{m, \delta}(S)  \; \mbox{or} \;   V_{m, \delta}\; \mbox{when $H$ or $(S,H)$ are understood),} 
\end{equation*} which is the parameter space for the universal family of irreducible curves in $|mH|$ 
having only $\delta$ nodes as singularities; this is called the $(m,\delta)$--{\em Severi variety} 
of $(S,H)$. 

We will assume that there exists a Deligne--Mumford {\em moduli stack} $\mathcal B$ parametrizing isomorphism classes of polarized surfaces $(S,H)$ as above. Since we will basically deal only with local properties, we can get rid of the stack structure. 
Indeed, up to replacing $\mathcal B$ with an \'etale  finite type representable cover, we may pretend that $\mathcal B$ is a 
{\em fine moduli scheme}.  Although not necessary, we will assume that $\mathcal B$ is irreducible (otherwise one may replace $\mathcal B$ with one of its components). 

Then we may consider the scheme  $\V_{m,\delta}$, called the $(m,\delta)$--\emph{universal Severi variety} over $\mathcal B$, which is endowed with a morphism 
\begin{equation*}\label{eq:phiVK2}
\phi_{m,\delta}: \V_{m,\delta}\to \mathcal B, 
\end{equation*}whose fiber over $(S,H)\in \mathcal B$ is $V_{m, \delta} (S,H)$. A point in 
$\V_{m,\delta}$ can be identified with a pair $(S,C)$, with $(S,H)\in \mathcal B$ and $C\in V_{m, \delta} (S,H)$.

We make the following:

\begin{ass}\label{ass:Sevvaruniv}  
\begin{inparaenum} [(i)]
\item $\mathcal B$ is smooth; \\ 
\item  for all  $(S,H)\in \mathcal B$, the surface  $S$ is {\em regular}, i.e., $h^ 1(S,\O_S)=0$,
and $h^ 0(S, \mathcal T_S)=0$, i.e., $S$ has no positive dimensional automorphism group;\\
\item for any $m \geqslant1$ and  $\delta\in \{0,\ldots,\ell (m)\}$ and for all  $(S,H)\in \mathcal B$, $\ell(m)$ is constant and the Severi variety $V_{m, \delta} (S,H)$ is smooth, of pure (and {\em expected}) dimension $\ell (m)-\delta$ (hence $\V_{m,\delta}$ is smooth, of pure dimension $\dim(\mathcal B)+\ell (m)-\delta$, and $\phi_{m,\delta}$ is smooth and surjective). \\
\end{inparaenum} 
\end{ass}

\begin{remark}\label{rem:casi} In Assumption \ref {ass:Sevvaruniv}(i) we could have asked $\mathcal B$ to be \emph{generically} smooth, and in (iii) we could have asked that for the \emph{general} $(S,H)\in \mathcal B$, the Severi variety $V_{m, \delta} (S,H)$ is smooth, of pure dimension $\ell (m)-\delta$. But under these weaker assumptions, (i) and (iii) hold on a Zariski dense open subset. Since we will be interested only in what happens at the general point of $\mathcal B$, we may replace $\mathcal B$ with this open subset.  The hypotheses in (ii) are technical and  not strictly necessary for our purposes, but they make things easier for us. 

Conditions (i)--(iii) hold  in some important cases, e.g., for polarized $K3$ surface of genus $p$  (in which case the moduli stack is usually denoted by $\mathcal K_p$, is of dimension $19$, and $\ell(m)=p(m)$ for any $m \geqslant 1$, cf.,\,e.g.,\,\cite{fkps,halic2,halicb,cfgk}). Moduli spaces exist also for polarized Enriques surfaces (cf.\;\cite {GrHu}) and degenerations of (polarized) Enriques surfaces are also studied (cf.\;e.g. \cite{E,K,M,Shah}).

Another relevant class is the one of minimal, regular surfaces of general type $(S,H)$ whose moduli space has at least one (generically) smooth component $\mathcal B$ with points $(S,H)$ verifying, for some $m$ and $\delta$, the conditions in \cite{CS,F1,F2} ensuring smoothness and expected dimension of any component of $V_{m,\delta}(S,H)$.  Particular cases are, for some $m$ and $\delta$, surfaces  in $\PP^3$ of degree $d \geqslant 5$ (cf. \cite{CC}) and  complete intersections of general type in $\PP^N$. 
\end{remark}

Consider now the {\em moduli map} 
\begin{equation*}\label{eq:modmap2}
\psi_{m,\delta} :\V_{m,\delta} \to {\mathcal M}_{g}, \,\, \text{where}\,\,  g= p(m)-\delta,  
\end{equation*}
and where ${\mathcal M}_g$ denotes the moduli space of smooth, genus--$g$ curves: $\psi_{m,\delta}$ sends a curve to the isomorphism class of its normalization. 
In this set--up, one is interested in the following general problem:  find conditions on $m$ and $\delta$ ensuring the existence of a component $\V$ of $\V_{m, \delta}$ such that $\psi_{m,\delta}|_{\V}$ is either {\em generically finite} onto its image or {\em dominant} onto ${\mathcal M}_g$. By taking into account Assumption \ref {ass:Sevvaruniv}, in principle,
one may expect
\begin{equation*}
\begin{array}{ll}
\mbox{{dominance if}} & \dim(\mathcal B)+\ell(m)-\delta \geqslant \dim ({\mathcal M}_{g}), \\
\mbox{{generic finiteness onto its image if}} &\dim(\mathcal B)+\ell(m)-\delta \leqslant \dim ({\mathcal M}_{g}).   
\end{array}
\end{equation*}

The typical example is the case of polarized $K3$ surfaces studied by various authors (cf.,\;e.g.,\;\cite{MM,fkps,halic2,halicb,Ke,cfgk}).  In particular,  \cite{cfgk} and Kemeny in \cite{Ke} independently  show that, as expected, $\psi_{m,\delta}$ is generically finite on some component for all $g=p(m)-\delta \geqslant11$ with only a few finite possible exceptions $(m,g)$ with $m \leqslant4$, for fixed $p$; moreover \cite{cfgk} shows that $\psi_{m,\delta}$ is dominant, as expected, for $g\leqslant 11$ with only a few finite possible   exceptions $(m,g)$ with $m \leqslant4$, for fixed $p$. The precise result for $m=1$ is the following:

\begin{thm} \label{thm:mappamod}Let $\V_{m, \delta}$ be the universal Severi variety over the moduli space $\mathcal K_p$ of polarized $K3$ surfaces $(S,H)$ of genus $p$. For $m=1$  and $g=p-\delta$ one has:\\
\begin{inparaenum}[(a)]
\item  \cite{cfgk,Ke} if $g \geqslant15$ there is  a component 
$\V$ of $\V_{1, \delta}$ such that ${\psi_{1,\delta}}_{|_{\V}}$ is generically finite onto its image;\\
\item \cite{cfgk} if $g\leqslant 7$ there is a component 
$\V$ of $\V_{1, \delta}$ such that ${\psi_{1,\delta}}_{|_{\V}}$ is dominant onto $\mathcal M_g$.
\end{inparaenum}
 \end{thm}

In case (a)  Kemeny's result is  stronger in the sense that he may weaken the assumptions on $g$ for infinitely many $p$'s.

The proofs in \cite{cfgk,Ke}, although different, both rely on studying the fibers of the moduli map on curves on {\it special} $K3$ surfaces. Kemeny's proof is inspired by ideas of \cite{MM}, and  uses appropriate curves on $K3$ surfaces with high rank Picard group. The approach in \cite{cfgk} is by specialization to a reducible $K3$ surface in a partial compactification of $\K_p$ and therefore uses an extension of the moduli map to an appropriate partial compactification of the Severi variety containing reducible curves, with target space $\overline{ \M}_g$. 

In the present paper we want to present a different approach to the aforementioned general problem. This approach relies on two different techniques. Firstly, it is based on the analysis of first order deformations of pairs 
$(S,C)\in \V_{m,\delta}$ as in \cite[\S\,4]{fkps}. 
The strategy in \cite{fkps}, which requires Assumption \ref {ass:Sevvaruniv}, was originally introduced for polarized $K3$ surfaces and for $m=1$, but can be easily adapted to $m\geqslant 1$ and to the case where the canonical bundle is not necessarily trivial 
(cf. also \cite[Thm. 1.1(ii)]{halic2}).  The upshot is the following. 
In the above setting,  take $(S,C) \in \V_{m,\delta}$ and set $Z := {\rm Sing} (C)$. Arguing as in \cite[\S\S\;4-5]{fkps}, 
the differential ${\rm d} \psi_{(S,C)}$ of $\psi:= \psi_{m,\delta}$ at $(S,C)$ can be identified with a suitable cohomology map  (the $H^1(\tau)$ in \cite[(4.21)]{fkps}). In particular, if $\mu_Z:\tilde S\to S$ is the blowing-up of $S$ at $Z$, $\tilde C$ is the strict transform of $C$ and $\T_{S}$ and $\T_{\tilde C}$ are the tangent bundles of $S$ and $\tilde C$ respectively, then
$$\coker ({\rm d} \psi_{(S,C)}) \cong H^2(\mu_Z^*(\T_S)(-\tilde C)) \;\; \mbox{and}  \;\; \ker ({\rm d} \psi_{(S,C)}) \cong H^1(\mu_Z^*(\T_S)(-\tilde C))/H^0(\T_{\tilde C}).$$
Moreover, by the Serre duality theorem and Leray isomorphism, as in \cite[Proof of Thm (5.1)]{fkps}), one has that 
$$
H^2(\mu_Z^*(\T_S)(-\tilde C))\simeq H^0(\Omega_S(mH+K_S) \* \I_{Z/S})\;\; \mbox{and}  \;\;H^1(\mu_Z^*(\T_S)(-\tilde C))\simeq
H^1(\Omega_S(mH+K_S) \* \I_{Z/S}).
$$
{\it By summarizing, if  $(S,C)$ belongs to a component $\mathcal V$ of $\mathcal V_{m, \delta}$, the map ${\psi_{m,\delta}}_{|\mathcal V}$ is
\begin{equation}\label{eq:modcond2}
\begin{array}{ll}
\mbox{{dominant}} & \mbox{if and only if} \;\;\;  h^0(\Omega_S(mH+K_S) \* \I_{Z/S}) = 0, \\
\mbox{{generically finite}} & \mbox{if} \;\;\; \;\;\;\;\;\;\;\;\;\;\;\;\;\;\;\;\;\;     h^1(\Omega_S(mH+K_S) \* \I_{Z/S}) = 0.   
\end{array}
\end{equation}
\noindent Finally, the vanishing of $H^1(\Omega_S(mH+K_S) \* \I_{Z/S})$ is equivalent to generic finiteness of ${\psi_{m,\delta}}_{|\mathcal V}$
if $g\geq 2$.}

Unfortunately, the vanishings  \eqref{eq:modcond2} are, in general, not so easy to be proved, even if one assumes $\delta=0$ (cf. \cite{beau}).  
This is where the second tool of our approach enters the scene (see \S \ref {S:logcomplex}). In order to prove the above vanishings, we propose to use degenerations. We assume in fact that the surfaces in $\mathcal B$ and nodal curves on them possess good semistable degenerations with limiting surfaces that are reducible in two components (the more general case of reducibility in more components could be treated, but, for simplicity, we do not dwell on this here). Then  we look at the \emph{limits} of the relevant cohomology spaces. The latter are driven by the so--called \emph{abstract log complex} (see \cite[\S\;3]{fri}). Using this we arrive at sufficient conditions for the vanishings in  \eqref{eq:modcond2}  to hold, expressed in terms of cohomological properties of suitable sheaves of forms on the two components of the limit surface (cf. \S\;\ref{ss:degdiff}). 
These properties are hopefully easier to prove than the vanishings  in \eqref{eq:modcond2}, since these components are simpler than $S$. 
The results are summarized in Proposition \ref{thm:main}, which is the main result of this note. 

In the rest of the paper (i.e. \S\S \ref {S:K3case} and \ref {sec:lemma}) we test our approach in  the (known) case of $K3$ surfaces for $m=1$, giving a new proof of Theorem \ref {thm:mappamod}. We do not claim that this is  easier than the proofs in \cite {cfgk,Ke}, but it works quite nicely and gives good hopes to fruitfully apply the same method in other unexplored cases, like the ones mentioned at the end of Remark \ref {rem:casi}.  We also mention that the approach of this paper can be applied to the $m>1$ case (using a slightly different degeneration, namely the one in \cite{ck}), but we leave this out as the bounds we obtain depend linearly on $m$ and are thus considerably weaker than the bounds in \cite {cfgk,Ke}. Furthermore, although the present approach gives the same result as \cite{cfgk} for $m=1$, the analysis of this case in \cite{cfgk} is finer (as it studies the family of degenerate $K3$s on which the curves in the fibers of the moduli map live) and is needed in the proof of the $m>1$ case. Thus, the approach in this paper cannot replace the proof of the $m=1$ case in \cite{cfgk}.

%
%

\subsubsection*{{\bf Terminology and conventions}}\label{ss:term} 
We work over $\mathbb{C}$. For $X$ any Gorenstein, projective variety, we denote by $\O_X$ and $\omega_X \simeq\O_X(K_X)$ the  structural and the canonical line bundle, respectively, where $K_X$ is a canonical divisor. We denote by $\T_X$ and $\Omega_X$ the tangent sheaf and the sheaf of $1$-forms on $X$, respectively. For $Y \subset X$ any closed subscheme, $\I_{Y/X}$ (or simply $\I_{Y}$ if $X$ is intended) will denote its ideal sheaf whereas $\N_{Y/X}$ (same as above) its normal sheaf. We use  $\sim$ to denote linear equivalence of divisors. We often abuse notation and identify divisors with the corresponding line bundles, using the additive and the multiplicative notation interchangeably. Finally, we use the convention that 
if $\F$ is a sheaf on a scheme $X$ and $Y \subset X$ is a subscheme, then $H^0(\F)_{|Y}$ is the image of the restriction map $H^0(\F) \to H^0(\F \* \O_Y)$.

\subsection*{Acknowledgements} The first three authors  have been supported by the GNSAGA of Indam and by the PRIN project ``Geometry of projective varieties'', funded by the Italian MIUR.

%
%

\section{Semistable degenerations and the abstract log complex}\label{S:logcomplex}

 In this section we will provide a tool for proving the vanishings of the cohomology groups occurring in \eqref{eq:modcond2} by degeneration of the surface and semicontinuity. The main results are summarized in Proposition \ref{thm:main} below.

\subsection{Semistable degenerations} \label{sec:SS} 

We recall some basic facts concerning semistable degenerations of compact complex surfaces and the associated abstract log complex (see \cite[\S\;3]{fri}). 
This complex allows to define flat limits of the sheaves occurring in \eqref{eq:modcond2}.

\begin{definition}\label{deformation}\begin{inparaenum}[(i)]
\item Let $R$ and $\Rc$ be connected, complex analytic varieties. Let $\Delta=\{t\in\mathbb C | \mid t\mid<1\}$. 
A proper, flat morphism $\alpha:\Rc\to\Delta$  is said to be a {\rm deformation of $R$} if 
$R$ is the scheme theoretical fibre of $\alpha$ over $0$. Accordingly, $R$ is said to be a {\rm flat limit} of $R_t$, the scheme theoretical fibre of $\alpha$ over $t \neq 0$. \\
\item If there is a line bundle $\H$ on $\Rc$, set  $H:=\H_{|_R}$ and $H_t:=\H_{|_{R_t}}$ for $t \neq 0$. Then the pair $(R,H)$ is  said to be a {\rm limit} of $(R_t,H_t)$ for $t \neq 0$. \\
\item The deformation (or, equivalently, the degeneration) is \emph{semistable} if  $\Rc$ is smooth (so that we may assume that  $R_t$ is smooth for $t\neq 0$) and $R$ has at most normal crossing singularities. \\
\item  Assume that there are $\delta$ disjoint sections $s_1, \ldots , s_{\delta}$ 
of $\alpha$ and that their images $\Z_1,\ldots, \Z_\delta$ are smooth curves in $\Rc$, disjoint from $\Sing (R)$.
Set $\mathcal Z := \bigcup_{i=1}^{\delta} \Z_i$.
Then we say that $Z:=\Z_{|_R}$ is a {\rm limit} of $Z_t:=\Z_{|_{R_t}}$ for $t \neq 0$.\\
\item  If  conditions (i)-(iv) are satisfied, we will say that $(R,H, Z)$ is a {\rm semistable degeneration} of $(R_t, H_t, Z_t)$, for $t \neq 0$, or that  $(R_t, H_t, Z_t)$, for $t \neq 0$, {\rm admits the semistable degeneration} $(R,H, Z)$. 
 \end{inparaenum}
\end{definition}

Let $\mathfrak R \stackrel{\alpha}{\longrightarrow}\Delta$ be a semistable degeneration of surfaces as in Definition \ref {deformation}.  Assume that all  components of R are smooth and R has no triple point. Then $\Sing(R)$ consists of the transversal intersection points of pairs of components of $R$. 

\medskip 

Consider the sheaf $\Omega_{\mathfrak R/\Delta}(\log R)$ on
$\mathfrak R$ defined by the exact sequence (cf. \cite{ste} and \cite[\S\;3.3]{del})
\begin{equation}\label{eq:omega-log-relative}
\xymatrix{
0\ar[r] & \alpha^*(\Omega_{\Delta}(0))\ar[r]^{\iota} &  \Omega_{\mathfrak R}(\log R)\ar[r] &  
\Omega_{{\mathfrak R} / \Delta}(\log R)\ar[r] & 0.}
\end{equation}
The map $\iota$ in \eqref{eq:omega-log-relative} has rank one at every point, whence $\Omega_{{\mathfrak R} / \Delta}(\log R)$ is locally free of rank $2$,
as recalled in the following remark.
\begin{remark}\label{rem:local} Away from $R$, one has $ \Omega_{\mathfrak R}(\log R)\cong \Omega_{\mathfrak R}$ and 
$\Omega_{{\mathfrak R} /\Delta}(\log R)\cong \Omega_{{\mathfrak R} / \Delta}$. In particular
\begin{equation*}\label{eq:Omegarest}
\Omega_{{\mathfrak R} / \Delta}(\log R)\otimes
\O_{R_t}\simeq \Omega_{{R_t}}, \,\, \text{for any}\,\, t\neq 0.
\end{equation*}

\noindent Let  $P\in R-\Sing(R)$,  and let $x,\;y,\;z$ be local coordinates on ${\mathfrak R}$ around $P$. 
Let $t$ be the coordinate on $\Delta$ and assume that $\alpha$ is locally defined by $t=x$ around $p$. Then 
$ \Omega_{\mathfrak R}(\log R)$ is locally free generated by $\frac{dx}{x},\;dy,\;dz$, the map $\iota$ is defined by 
$$\frac{dt}{t} \longrightarrow \frac{dx}{x}$$ hence $\Omega_{{\mathfrak R} / \Delta}(\log R)$ is locally free generated by  
$dy,\;dz$ (cf.\;\cite[Prop.\;2.2.c]{ev}). 

\noindent Let now $P\in \Sing(R)$.  From our assumptions, we may assume that $\alpha$ is locally defined by 
$t=xy$. Then $ \Omega_{\mathfrak R}(\log R)$ is locally free generated by $\frac{dx}{x},\;\frac{dy}{y},\;dz$, 
the map $\iota$ is defined by$$\frac{dt}{t} \longrightarrow \frac{dx}{x}+\frac{dy}{y}$$ and so  
$\Omega_{{\mathfrak R} / \Delta}(\log R)$ is locally free generated by  $\frac{dx}{x}=-\frac{dy}{y},\;dz$. 
\end{remark}

Following \cite[\S\;3]{fri}, we set 
\begin{equation*}\label{lambda-uno}
\Lambda_R^1:=\Omega_{{\mathfrak R} /\Delta}(\log  R)|_R
\end{equation*} which is locally free on $R$. 

\begin{lemma}\label{prop:semic} Let $(R,H,Z)$ be a semistable degeneration of $(R_t,H_t,Z_t)$ for $t\neq 0$, as in Definition \ref{deformation}.   Assume furthermore that all components of $R$ are smooth and $R$ has no triple points. Then for all $m\in \mathbb Z$, $i\in \mathbb N$ and $t\neq 0$, one has
\begin{equation*}\label{eq:semicont}
h^i (\Lambda^1_R \* \O_R (m H + K_R) \* \I_{Z/R}) \geqslant h^i(\Omega_{R_t} \* \O_{R_t}(mH_t+ K_{R_t})\* \I_{Z_t/R_t}). 
\end{equation*}
\end{lemma}

\begin{proof} The statement  follows by semicontinuity as $\Omega_{{\mathfrak R}/\Delta}(\log R)$ is flat over $\Delta$ and the ideal sheaf $\I_{\mathcal Z/\Rc}$ is flat over $\Delta$ by \cite[Prop.\;4.2.1(ii)]{Ser}. \end{proof}

%
%

\subsection{Degenerations of differentials}\label{ss:degdiff} From now on we assume that $R:= R_1 \cup R_2$, with $R_1,R_2$  smooth, projective surfaces, with transversal intersection along a smooth, irreducible curve $E:=R_1 \cap R_2$.  Then
\begin{equation}\label{eq:adj2}
{K_R}_{|_{R_i}} = (K_{\mathfrak R}+R)_{|_{R_i}}= (K_{\mathfrak R}+R_1+R_2)_{|_{R_i}}=
K_{R_i} + E, \; i = 1,2.  
\end{equation} In this situation there are exact sequences involving $\Lambda^1_R \* \O_R (m H + K_R) \* \I_{Z/R}$, which allow us to compute its cohomology by conducting computations on $R_1$ and $R_2$. 

Consider the exact sequences 
\begin{equation}  \label{eq:es1}
 \xymatrix{ 0  \ar[r] & \Omega_{R_i} \ar[r] & \Omega_{R_i} (\log E) \ar[r]^{\hspace{0.6cm}\rho_i} & \O_E \ar[r] & 0},  \;\;\;\; i = 1,2,  
\end{equation}
and
\begin{equation}\label{eq:es2}
 \xymatrix{ 0  \ar[r] & \Omega_{R_i}(\log E) \* \O_{R_i}(-E) \ar[r] & \Omega_{R_i} \ar[r]^{r_i} & \omega_E \ar[r] & 0,}\;\;\;\; i = 1,2,  
\end{equation}where $\rho_i$ is the \emph{residue map} and $r_i$ is the \emph{trace map} of differential forms, cf. \cite[\S\;2]{ev}. 

For the reader's convenience, we recall how $\rho_i$ and $r_i$ are defined locally around a point of $E$. We may assume that locally, in an open subset of $\mathbb C^4$ with coordinates $x,y,z,t$, the equation of  $\Rc$ is 
$xy-t=0$. We let $R_1$ be given by $x=t=0$ and $R_2$ by $y=t=0$, so that $E$ is given by $x=y=t=0$. In this chart
$$\O_{R_1} \cong \CC[[y,z]], \; \Omega_{R_1} \cong
\CC[[y,z]] \; dy \+ \CC[[y,z]] \; dz,$$$$\Omega_{R_1}(\log E) \cong
\CC[[y,z]] \; \frac{dy}{y} \+ \CC[[y,z]] \; dz, \; \O_{E} \cong \CC[[z]]\; {\rm and} \; \omega_E \cong \CC[[z]] \; dz,$$
\[  \rho_1\Big(f(y,z) \frac{dy}{y} + g(y,z)  \; dz\Big)  :=  f(0,z), \qquad r_1\Big(f(y,z) \; dy + g(y,z)  \; dz\Big)  : =  g(0,z)\; dz\]
and similarly for $R_2$.

Let $a: R_1 \sqcup R_2 \to R$ be the desingularization of $R$; consider the exact sequence 
\begin{equation}\label{eq:Afascio}
\xymatrix{
0   \ar[r] &  \mathcal A  \ar[r]  &   a_* \Big(\Omega_{R_1}(\log E) \oplus \Omega_{R_2}(\log E)\Big) \ar[r]^{\hspace{1.9cm}\rho} &  \O_E \ar[r]  & 0,
} 
\end{equation}defining $\A$, where $\rho: = \rho_1 + \rho_2$. As in  \cite[Pf. of Lemma\;3.1,\;p.94-95]{fri}, $\Lambda^1_R$ fits in the exact sequence 
\begin{equation}\label{eq:Astronzo}
\xymatrix{
0   \ar[r] &    \Lambda^1_R \ar[r]  &  \mathcal A \ar[r]^{r} &  \omega_E \ar[r]  & 0,
} 
\end{equation}where $r$ is locally defined as follows:  if $$(\omega_1,\omega_2) = \Big( f_1(y,z) \frac{dy}{y} + g_1(y,z) dz,\; f_2(x,z) \frac{dx}{x} + g_2(x,z) dz\Big)$$is a local section of $\mathcal A$, then 
$$r(\omega_1, \omega_2) := \big(g_1(0,z) - g_2(0,z)\big) dz.$$

\begin{remark}\label{rem:locmanonglob} {\normalfont Notice that 
\begin{equation} \label{locmanonglob}
a_*(\Omega_{R_1} \+ \Omega_{R_2}) \hookrightarrow \mathcal A \; \;  \mbox{and} \; \; 
r|_{a_*(\Omega_{R_1} \+ \Omega_{R_2})}  = r_1 - r_2,
\end{equation}
 where $r_1$ and $r_2$ are the maps in \eqref{eq:es2}. One may verify that the map $r_i$ does not extend to a map on $\mathcal A$.
}
\end{remark}

Having in mind \eqref{eq:modcond2} and Lemma \ref{prop:semic}, we are looking for conditions ensuring  the vanishing of 
$h^i (\Lambda_R^1 \otimes \O_R(mH + K_R) \otimes \I_{Z/R}) $, for $i =0,1$. 
To this end, consider the maps 
\begin{eqnarray*}
\FF := \rho \otimes \O_R(mH + K_R) \otimes \I_{Z/R}, & & q := r \otimes \O_R(mH + K_R) \otimes \I_{Z/R}, \\
\FF_i := \rho_i \otimes \O_{R_i}(mH + K_{R_i} + E ) \otimes \I_{Z_i}  & \mbox{and} & q_i : = r_i  \otimes \O_{R_i}(mH + K_{R_i} + E ) \otimes \I_{Z_i}.
\end{eqnarray*}Then by  \eqref{eq:es1}, \eqref{eq:es2} and the fact that $Z_i$ avoids $E$, for $i=1, 2$ we have 
\begin{equation} \label{eq:C2}\footnotesize
\xymatrix{ &  & 0 \ar[d] & & \\
0   \ar[r] &  \Omega_{R_i}(\log E)(mH+K_{R_i})  \* \I_{Z_i}  \ar[r]  &   \Omega_{R_i}(mH+K_{R_i} + E)  \* \I_{Z_i}  \ar[r]^{\hspace{0.6cm}q_i} \ar[r] \ar[d] & \omega_E(mH+K_{R_i} + E) \ar[r]  & 0 \\
           &     & \Omega_{R_i}(\log E)(mH+K_{R_i} + E) \* \I_{Z_i}  \ar[d]^{\FF_i}  \ar[d] &   &  \\
&   & \O_E(mH+K_{R_i} + E) \ar[d] & & \\
&  & 0 & & 
} 
\end{equation} Moreover, by \eqref{eq:Afascio}--\eqref{locmanonglob}, we have

\begin{equation} \label{eq:C22}\footnotesize
\xymatrix{ 0   \ar[r] & \A \* \O_R(mH+K_R) \* \I_{Z/R} \ar[r] & a_* \Big(\+_{i=1}^2 \Omega_{R_i}(\log E)(mH+K_{R_i}+E) \* \I_{Z_i} \Big) \; \;  \ar[r]^{\hspace{2cm}\FF=\FF_1+\FF_2} \; \; &  \O_E(mH+K_R) \ar[r]  & 0 \\
 & & a_* \Big(\+_{i=1}^2  \Omega_{R_i}(mH+K_{R_i}+E) \* \I_{Z_i} \Big) \ar@{^(->}[u] \ar@{_(->}[d] \ar[dr]^{q_1-q_2}& & \\
0   \ar[r] &    \Lambda^1_R (mH+K_R) \*  \I_{Z/R}  \ar[r]  &  \mathcal A \* \O_R(mH+K_R) \* \I_{Z/R} \ar[r]^{\hspace{0.7cm}q} &  \omega_E(mH+K_R) \ar[r]  & 0.
} 
\end{equation}

 The next two lemmas provide sufficient conditions for the vanishings of the two first cohomology groups of $\Lambda_R^1(mH + K_R) \otimes \I_{Z/R}$.  

\begin{lemma}\label{lem:condsuff2} Assume that 
\begin{equation}\label{eq:neuer2} 
\im (H^0(\FF_1)) \cap \im (H^0(\FF_2)) =\{0\}.
 \end{equation} Then $h^0(\Lambda^1_R (mH + K_R) \otimes \I_{Z/R}) = 0$ if and only if
 \begin{equation}\label{eq:neuer4} \im (H^0(q_1)) \cap \im (H^0(q_2))= \{  0 \} 
\end{equation} and
\begin{equation}\label{eq:neuer3} h^0(\Omega_{R_i} (\log E) (mH+K_{R_i}) \* \I_{Z_i})=0\,\,\, \text{for}  \; \; i=1,2.
\end{equation}
\end{lemma}

\begin{proof} By \eqref{eq:C22}, we have that $h^0 (\Lambda^1_R  (mH + K_R) \otimes \I_{Z/R}) = 0$ if and only if
$H^0(q)$  is injective. Moreover, as $H^0(\FF) = H^0(\FF_1) + H^0(\FF_2)$, we get from  \eqref{eq:neuer2}, \eqref{eq:C22} and the vertical sequence in \eqref{eq:C2} that 
$$H^0(\mathcal A \otimes \O_R(mH + K_R) \otimes \I_{Z/R}) = \ker(H^0(\FF)) = \bigoplus_{i=1}^2 \ker (H^0(\FF_i)) = \bigoplus_{i=1}^2 H^0(\Omega_{R_i}(mH+K_{R_i} + E) \* \I_{Z_i}).$$ Hence, by \eqref{eq:C22} again, we have $H^0(q) = H^0(q_1) + H^0(q_2)$.
The statement now follows by the isomorphism $\ker(H^0(q))\simeq\Big(\oplus_i\ker(H^0(q_i)\Big)\bigoplus  \Big(\im (H^0(q_1)) \cap \im (H^0(q_2))\Big)$
and the horizontal sequence in \eqref{eq:C2}.
\end{proof}

\begin{lemma} \label{lem:condsuff} 
 Assume that 
\begin{equation}
  \label{eq:buffon1}
 \sum_{i=1}^2 H^0 (q_i) \; \; \mbox{is surjective}, 
\end{equation}
\begin{equation}
  \label{eq:buffon2}
\sum_{i=1}^2 H^0 (\FF_i) \; \; \mbox{is surjective}
\end{equation}
and
\begin{equation}
  \label{eq:buffon3}
h^1(\Omega_{R_i}(\log E)(mH+K_{R_i} + E) \* \I_{Z_i}) = 0, \; \; \mbox{for $i=1,2$}.\end{equation}
Then 
$h^1(\Lambda^1_R  (mH + K_R) \otimes \I_{Z/R}) = 0$. 
 \end{lemma}

\begin{proof} This follows from \eqref{eq:C22}.
\end{proof}

We summarize our main results as follows:

\begin{prop} \label{thm:main} Let $S'$ be a smooth, projective surface, $H'$ a line bundle on $S'$ and 
$Z' \subset S'$ a reduced zero-dimensional scheme. Assume that $(S',H',Z')$ admits a  semistable degeneration $(R,H,Z)$ with $R=R_1\cup R_2$, 
 $R_1,R_2$  smooth, with transversal intersection $E=R_1 \cap R_2$, and $Z=Z_1\cup Z_2$, with $Z_i \subset R_i-E$, $i=1,2$. 

 If \eqref{eq:neuer2}, \eqref{eq:neuer4} and \eqref{eq:neuer3} hold,  one has $h^0(\Omega_{S'}(mH'+K_{S'}) \* \I_{Z'/S'}) =0$.

 If \eqref{eq:buffon1}, \eqref{eq:buffon2} and \eqref{eq:buffon3} hold,  one has $h^1(\Omega_{S'}(mH'+K_{S'}) \* \I_{Z'/S'}) =0$. 
\end{prop}

\begin{proof}  This follows from the two preceding lemmas and Lemma \ref{prop:semic}. 
\end{proof}

%
%
\section{The $K3$ case}\label{S:K3case}

 In the rest of the note, we will show how  to apply  Proposition \ref{thm:main}  to  semistable degenerations of  smooth, primitively polarized $K3$ surfaces,  thus giving, via  \eqref{eq:modcond2},  a new proof of Theorem \ref{thm:mappamod}.

\subsection{A semistable degeneration of K3 surfaces} This degeneration is well known (see \cite {clm}) and we recall it to fix notation.

Let $p=2n+\varepsilon \geqslant3$ be an integer, with $n \geqslant 1$ and $\varepsilon\in \{0,1\}$, and let 
$E' \subset \PP^p$ be a smooth, elliptic normal curve of degree $p+1$.  Consider two general line bundles $L_1, L_2 \in \Pic^2(E') $ with $L_1\neq L_2$. In particular there is no relation between $L_1,L_2$ and $\O_{E'}(1)$ in $\Pic(E')$. 

We denote by $R'_1$ and $R'_2$ the rational normal scrolls of degree $p-1$ in $\PP^p$ described by the secant lines of $E'$ generated by the divisors in $|L_1|$ and $|L_2|$, respectively. We have 
\[ R'_i \cong  \mathbb{F}_{1-\varepsilon} = 
\begin{cases} 
\PP^1 \times \PP^1 &  \; \mbox{if $p=2n+1$},  \\ 
 \mathbb{F}_1 &  \; \mbox{if $p=2n$}
\end{cases} 
\]($\mathbb{F}_{1-\varepsilon}$ is called the \emph{type} of the scrolls $R'_1$ and $R'_2$) and $R'_1$ and $R'_2$ transversely intersect along $E'$, which is anticanonical on both (cf.\,\cite{clm}). 

Denoting by $\mathfrak{p}_i: R_i' \to \PP^1$ the structural morphism  and by $\sigma_i$ and $F_i$ 
a section with minimal self--intersection and a fiber of $\mathfrak{p}_i$, respectively, we have 
$\sigma_i^2 = \varepsilon -1$, $\sigma_i \cdot F_i = 1,$ $F_i^2 = 0$ and $\Pic(R_i') \cong \mathbb{Z} [\sigma_i] \oplus \mathbb{Z} [F_i]$. 
One has 
   \begin{equation}\label{eq:A}
\O_{R'_i}(1) \simeq \O_{R'_i}(\sigma_i + nF_i)\,\,\, \text{and}\,\,\, K_{R_i'} \sim - 2 \sigma_i - (3-\varepsilon) F_i\sim -E', 
\end{equation}
and $\Omega_{R_i'}$ fits in the exact sequence
\begin{equation}\label{eq:A1.1}
\xymatrix{ 
0 \ar[r] & \mathfrak{p}_i^*(\omega_{\PP^1})\cong \O_{R_i'} (-2 F_i) \ar[r] & \Omega_{R_i'} \ar[r] & \Omega_{R_i'/\PP^1} \cong  \O_{R_i'} (- E' + 2 F_i) \ar[r] & 0,
}
\end{equation}which splits if $\varepsilon=1$. 

Set  $R':= R_1' \cup R_2'.$  The \emph{first cotangent sheaf} $T^1_{R'}$ (cf. \cite[\S\;1]{fri})  is the degree $16$ line bundle on $E'$
 \begin{equation}
  \label{eq:B}
T^1_{R'}\cong \N_{E'/R'_1}\otimes \N_{E'/R'_2}\cong \O_{E'}(4) \otimes (L_1 \otimes L_2)^ {\otimes (3-p)},
\end{equation} the last isomorphism coming from \eqref{eq:A}. 

The Hilbert point of  $R'$ sits in the smooth locus of the component $\mathcal X_p$ of the Hilbert scheme whose general point represents a smooth $K3$ surface of degree $2p-2$ in $\PP^p$ having Picard group generated by the hyperplane section 
(cf.\,\cite[Thms.\,1,\,2]{clm}).
 
The fact that $T^1_{R'}$ is non-trivial on $E'$ implies that $R'$ does not admit any semistable deformation (cf. \cite[Prop.\,1.11]{fri}). Indeed, the total space of a general  flat deformation of $R'$ in $\mathbb P^p$ is singular along $16$ points on $E'$ 
that are the zeros of a global section of $T^1_{R'}$ (cf.\,\cite[\S\;2]{fri}). More precisely
if $\Rc' \stackrel{\alpha'}{\longrightarrow} \Delta$ is a (general) embedded deformation of 
$ R'$ in $\mathbb P^p$ corresponding to a (general) section $\tau\in H^ 0(R', \mathcal N_{R'/\PP^ p})$, then the total space $\Rc'$ has  {double points} at the 16 (distinct) points of the divisor $W\in \vert T^1_{R'}\vert$, with 
\[
W:=\{s_\tau=0\} \quad \text{where}\quad \tau\in H^ 0(\mathcal N_{R'/\PP^ p})\overset{\kappa} {\longrightarrow} s_\tau:=\kappa(\tau)\in H^ 0(T^1_{R'}),
\]the map $\kappa$ being {surjective} (see \cite[Cor.\;1]{clm}).
 By blowing up $\Rc'$ along these singular points and  contracting every exceptional divisor on one of the two irreducible components of the strict transform of $R'$, one obtains
a small resolution of singularities $\Pi: \Rc \to \Rc' $  and a 
semistable degeneration $\Rc \stackrel{\alpha}{\longrightarrow}\Delta$ of $K3$ surfaces,  
with central fiber $R= R_1\cup R_2$, where $R_i=\Pi^{-1}(R'_i)$, $i =1,2$.  
Then $E:= R_1 \cap R_2 = {\rm Sing}(R)$ is such that $E'=\Pi(E) \cong E$ and $T^1_{R}\simeq \O_E$.  The curve $E'$ [resp.~ $E$] is anticanonical on $R'_i$ [resp.~ on $R_i$], for $i=1,2$, hence both $R'$ and $R$ have trivial dualizing sheaf (see \cite[Rem.\,2.11]{fri}).  
On $\Rc'$ there is a line bundle $\H'$ restricting to the hyperplane bundle on each fiber. We set $\H=\Pi^{*}(\H')$. 

The map $\pi_i:=\Pi_{|R_i}:R_i \to R'_i$ is the contraction of $k_i$  disjoint $(-1)$-curves 
$\ee_{i,1},\ldots, \ee_{i,{k_i}}$, such that $\ee_{i,j} \cdot E=1$, to distinct points $x_{i,1},\ldots, x_{i,k_i}$ on $E'$.  We set $W_i=x_{i,1}+\cdots+ x_{i,k_i}$ and $\ee_i:=\sum_{j=1}^{k_i}\ee_{i,j}$, for $i =1,2$.
Then $W= W_1 + W_2 \in |T^1_{R'}|$ is general, hence reduced. 

If $\Rc' \stackrel{\alpha'}{\longrightarrow} \Delta$ is general in the above sense, we will accordingly say that $\Rc \stackrel{\alpha}{\longrightarrow}\Delta$ is \emph{general} and $(R_t,H_t)$, for $t\neq 0$, can be thought of as the general point of $\mathcal K_p$.

\subsection{Technical lemmas} 
We will now develop tools to verify the conditions \eqref{eq:neuer2}-\eqref{eq:neuer3} and \eqref{eq:buffon1}-\eqref{eq:buffon3}  in Proposition \ref{thm:main}. 


Consider the relative cotangent sequence of the map $\pi_i: R_i\to R'_i$, together with the dual of the exact 
sequence defining its normal sheaf $\N_{\pi_i}$ (c.f. e.g. \cite[Ex.\;3.4.13(iv)]{Ser}); we have $\Omega_{R_i/R'_i} \simeq  \Shext^1(\N_{\pi_i}, \O_{R_i}) \simeq  \Shext^1(\O_{\ee_i}(-\ee_i), \O_{R_i})$. 
One easily verifies that $\Shext^1(\O_{\ee_i}(-\ee_i), \O_{R_i})\cong \oplus_{j=1}^{k_i} \omega_{\ee_{i,j}}$, whence  
\begin{equation}\label{eq:Aexbl}
 \xymatrix{ 0  \ar[r] & \pi_i^*(\Omega_{R'_i}) \ar[r] & \Omega_{R_i} \ar[r] & 
\Omega_{R_i/R'_i}  \cong 
\+_{j=1}^{k_i} \omega_{\ee_{i,j}} \ar[r] & 0.}
\end{equation}

On each $(-1)$--curve $\ee_{i,j}$ on $R_i$, with $i=1,2$ and $j=1,\ldots, k_i$, we can consider the two points $Y_{i,j}$ and $Y'_{i,j}$ respectively cut out on $\ee_{i,j}$ by the strict transform on $R_i$ of the ruling of $R'_i$ through $x_{i,j}$ and by $E$. Note that $Y_{i,j}=Y'_{i,j}$ if and only if $x_{i,j}$ is a ramification point on $E$ of the linear series $|L_i|$. This will not be the case if $W\in |T^1_{R'}|$ is general. We will consider the 0--dimensional scheme $Y_i=\sum_{j=1}^ {k_j} Y_{i,j}$ on $R_i$, for $i=1,2$. 
Since  $\pi_i(Y_{i,j})= \pi_i (Y'_{i,j})=x_{i,j}$, we have $W_i=\pi_i(Y_i)$, for $i=1,2$. 
Then:

\begin{lemma} \label{lemma:Aaiutino} Let $W \in |T^1_{R'}|$ be general. We have an exact sequence
\begin{equation}\label{eq:A2.10}
\xymatrix{ 
0 \ar[r] & \O_{R_i}(-2\pi_i^*(F_i)) \ar[r] & \Omega_{R_i} \ar[r] &  \O_{R_i}(-E+2\pi_i^*(F_i)) \* \I_{Y_i} 
 \ar[r] & 0.}
\end{equation}
\end{lemma}
\begin{proof} To simplify notation, we set $S=R_i, S'=R'_i, \pi=\pi_i$, $F=F_i$, $L=L_i$, for $i=1,2$. 
By the local nature of the claim, we may and will assume that $\pi : S \to S'$ is the blow--up at only one point $w \in E'$, which, by the generality assumption, is not  any of the four ramification points of the pencil $|L|$. Denote by $\ee$ the $\pi$--exceptional divisor. 

The cokernel of the injective map $\pi^*(\O_{S'}(-2F) )\to \pi^*(\Omega_{S'}) \to \Omega_{S}$ obtained from \eqref{eq:A1.1} and \eqref{eq:Aexbl} is torsion free with second Chern class $c_2=1$. For determinantal reasons, it must be equal to $\O_{S}(-E+2\pi^*(F)) \* \I_{Y}$ for some $Y \in \ee$. 
We will prove, with a local computation, that $Y$ is the intersection of $\ee$ with the strict transform on $S$ of the ruling of $S'$ passing through $w$. This will prove \eqref {eq:A2.10}. 

Choose a chart $U\subset S'$ centered at $w$ with coordinates $(x,z)$, such that the structural map $S\to 
\mathbb P^ 1$ is given on $U$ by $(x,z)\mapsto x$ and $E'$ has equation $z=0$ (we can do this because
 $w$ is not a ramification point of $|L|$).
Then ${\Omega_{S'}}_{|_U}$ is generated by $dx, dz$, the sheaf $\pi^*(\O_{S'}(-2F) )_{|U}$ is generated by $dx$ and $dz$ is the local generator for the quotient line bundle.

Consider $\widetilde{U} \subset U \times \PP^1$ the blow-up of $U$ at $w = (0,0)$. If $[\epsilon,\eta]$ are homogeneous coordinates on $\PP^1$, an equation for $\widetilde{U}$ in $U \times \PP^1$ is $x \eta = \epsilon z$. The open subset $\widetilde{U}_0$ where $\eta \neq 0$, sits in $U \times \mathbb{A}^1$ and has equation $x = tz$ there, where  $t = \frac{\epsilon}{\eta}$ is the coordinate on $\mathbb A^ 1$. So we have coordinates $(t,z)$ on $\widetilde{U}_0$ and the 
equation for $\ee_0 = \ee \cap  (U \times \mathbb{A}^1)$ is $z=0$, whereas $t=0$ is the equation of the strict transform of the ruling on $S'$ though $w$. 

In $\widetilde{U}_0$ we have $dx=t dz + z dt$ so the injection in \eqref{eq:A1.1} gives 
$$ 
\xymatrix{
0 \ar[r] &  \O_{\widetilde{U}_0}  \ar[r]^{\tiny{\left(\begin{array}{r}z\\ t \end{array}\right)}} & \O_{\widetilde{U}_0}^{\oplus 2} \ar[r]^{(-t \;\;\; z)} \ar[r] &
\O_{\widetilde{U}_0},
}
$$where:\\
\begin{inparaenum} [$\bullet$]
\item  $ \O_{\widetilde{U}_0}^{\oplus 2} \simeq \Omega_{\widetilde{U}_0}$ is generated by $dt, \, dz$;\\
\item  the inclusion $\O_{S} (- 2 \pi^*(F)) \hookrightarrow \Omega_{S}$   is locally given in $\widetilde{U}_0$ 
by $ 1 \mapsto z\,dt + t\,dz$, and\\
\item the image of the map $(-t \;\; z)$ is the ideal generated by $t$ and $z$ in $\O_{\widetilde{U}_0}$, 
\end{inparaenum} \newline hence $Y$ has equation $t=z=0$, as wanted. 
\end{proof}

For $i=1,2$, we denote by $X_i \in |2L_i|$ the ramification divisor of $|L_i|$ on $E'$, or, by abuse of notation, on $E$.

\begin{lemma}\label{lem:ramif} Same assumptions as in Lemma \ref{lemma:Aaiutino}. Then the composed map$$\O_{R_i}(-2\pi_i^*(F_i)) \hookrightarrow \Omega_{R_i} \stackrel{r_i}{\longrightarrow} \omega_E,$$given by \eqref{eq:es2}  and \eqref{eq:A2.10}, is non--zero, for $i=1,\,2$. Its image is the line bundle $\omega_E(-X_i)$. \end{lemma}
\begin{proof}  We use the same simplified notation as in the proof of Lemma \ref{lemma:Aaiutino}. Accordingly we will write $X$ for $X_i$. 

Take a point $P$ in $X$ on $E$. By the generality assumption on $W$, $R$ and $R'$ are isomorphic around $P$. So, if we work locally, we may do it on $R'$. Choose a chart $U$ on $R'$ centered at $P$, with coordinates $(x,z)$, such that the structural map $S\to 
\mathbb P^ 1$ is given on $U$ by $(x,z)\mapsto x$ and $E'$ has equation $x=z^ 2$. So we may take $z$ as the coordinate on $E'$. 

On $U$ the sheaf injection in \eqref{eq:A2.10} is the same as the one in \eqref{eq:A1.1} which, by the proof of Lemma \ref{lemma:Aaiutino}, is given by $1 \mapsto dx$. Composing this injection with the trace map on $E'$ gives 
$1 \mapsto dx = 2 z dz$, which shows that the map is non--zero and its image is a differential form on $E$ vanishing at $P$.  

The proof is accomplished by making similar local computation at points $P\in E$ which are not on $X$. This can be left to the reader.  
\end{proof}


By Lemmas \ref {lemma:Aaiutino} and \ref {lem:ramif}, we have the following commutative diagram

\begin{equation} \label{eq:vaffa}\footnotesize
\xymatrix{ 
& 0  \ar[d] & 0 \ar[d] & 0 \ar[d] & \\
0 \ar[r] & \O_{R_i}(mH-2\pi_i^*(F_i)-E) \* \I_{Z_i}\ar[r] \ar[d] & \Omega_{R_i}(\log E)(mH-E)\* \I_{Z_i}  \ar[r] \ar[d] &\O_{R_i}(mH-E+2\pi_i^*(F_i))\*  \I_{X_i \cup Y_i \cup Z_i} \ar[r] \ar[d] & 0 \\
0 \ar[r] &\O_{R_i}(mH-2\pi_i^*(F_i)) \* \I_{Z_i}\ar[r] \ar[d] & \Omega_{R_i}(mH)\* \I_{Z_i} \ar[r] \ar[d]^{q_i} & \O_{R_i}(mH-E+2\pi_i^*(F_i)) \* \I_{Y_i \cup Z_i}\ar[r] \ar[d] & 0 \\
0 \ar[r] & \omega_E(mH)(-X_i) \ar[r] \ar[d] & \omega_E(mH)  \ar[r] \ar[d] & \O_{X_i} \ar[r] \ar[d] & 0 \\
& 0 & 0 & 0 &
}
\end{equation}where the second vertical 
and horizontal exact sequences are respectively \eqref{eq:es2}, \eqref{eq:A2.10} tensored by $\O_{R_i}(mH) \otimes \I_{Z_i}$, and where we used the isomorphism $\O_E\simeq\omega_E$.

The composed injection 
$\xymatrix{ \O_{R_i}(mH-2\pi_i^*(F_i)) \* \I_{Z_i} \ar[r] & \Omega_{R_i}(mH) \* \I_{Z_i} \ar[r] & \Omega_{R_i}(\log E)(mH) \* \I_{Z_i}}$ obtained from the above diagram and \eqref{eq:es1} tensored by $\O_{R_i}(mH) \* \I_{Z_i}$ fits in the diagram
\begin{equation} \label{eq:nculo}\footnotesize
\xymatrix{ 
&  & 0 \ar[d] & 0 \ar[d] & \\
0 \ar[r] & \O_{R_i}(mH-2\pi_i^*(F_i))\* \I_{Z_i} \ar[r] \ar@{=}[d] & 
\Omega_{R_i}(mH) \* \I_{Z_i} \ar[r] \ar[d] &  \O_{R_i}(mH-E+2\pi_i^*(F_i)) \* \I_{Y_i \cup Z_i}  \ar[r] \ar[d] & 0 \\
0 \ar[r] & \O_{R_i}(mH-2\pi_i^*(F_i))\* \I_{Z_i} \ar[r]  & \Omega_{R_i}(\log E)(mH) \* \I_{Z_i} \ar[r] \ar[d]^{\FF_i} &  \O_{R_i}(mH+2\pi_i^*(F_i)) \* \I_{X_i \cup Y_i \cup Z_i}  \ar[r] \ar[d] & 0 \\
 &  & \O_E(mH)  \ar@{=}[r] \ar[d] &  \O_E(mH)  \ar[d] & \\
&  & 0 & 0 &
}
\end{equation}

 We want to  describe $\im (H^0(\FF_i))$ and $\im (H^0(q_i))$ in the case $m=1$. To this end we first define the following subspaces of $H^0(\O_E(H)) \cong H^0(\omega_E(H))$. Recalling convention and notation as at the end of the introduction, 
from the right-most vertical sequence in \eqref{eq:nculo}, we set
\begin{equation}
  \label{eq:defV}
V_i:=H^0(\O_{R_i}(H+2\pi_i^*(F_i)) \* \I_{X_i \cup Y_i \cup Z_i})_{|E} \subseteq H^0(\O_{E}(H)),  \; i=1,\,2.
\end{equation}
The inclusion works as follows: take a (non--zero) section $s\in V_i$ (which vanishes along a curve $C$ containing $X_i$), restrict it to $E$, then divide by fixed local equations of the points in  $X_i$ (i.e., remove $X_i$ from the divisor cut out by $C$ on $E$). 

Similarly, from the left-most vertical sequence in \eqref{eq:vaffa}, we define
\begin{equation}
  \label{eq:defU}
U_i:= H^0(\O_{R_i}(H-2\pi_i^*(F_i)) \* \I_{Z_i})_{|E} \subseteq H^0(\omega_E(H) (-X_i)) \subset H^0(\omega_{E}(H)), \,\, \text {for} \,\,  i =1,2. 
\end{equation}
At divisor level, the inclusion  is given by taking a divisor in 
$|\O_{R_i}(H-2\pi_i^*F_i) \* \I_{Z_i}|$, restricting it to $E$, and then adding the points $X_i$.

\begin{lemma} \label{lemma:VU}
  If $m=1$, then
\begin{equation}
  \label{eq:cU}
 \im (H^0(q_i) )=U_i \cong H^0(\O_{R_i}(H-2\pi_i^*(F_i)) \* \I_{Z_i}) 
\end{equation}
and
\begin{eqnarray}
  \label{eq:cV}
\im (H^0(\FF_i)) \sub V_i \cong H^0(\O_{R_i}(H+2\pi_i^*(F_i)) \* \I_{X_i \cup Y_i \cup Z_i}), \; \\ \nonumber \mbox{with equality if} \; h^1(\O_{R_i}(H - 2\pi_i^*(F_i)) \*  \I_{Z_i})=0. 
\end{eqnarray}
\end{lemma}

\begin{proof} This follows from \eqref{eq:vaffa} and \eqref{eq:nculo} with $m=1$, and $h^0(\O_{R_i}(H-E+a\pi_i^*(F_i)))=0$ for any integer $a$.
\end{proof}

\subsection{A new proof of Theorem \ref{thm:mappamod}}

We consider $V_{\delta}(R)$ the locally closed subscheme of the linear system $|H|$ on $R$ parametrizing the universal
family of curves $C \in |H |$ having only nodes as singularities, exactly $\delta$ of them (called the {\it marked nodes}) off the
singular locus $E$ of $R$, and such that the partial normalization of $C$ at 
the $\delta$ marked nodes is connected, i.e., the marked nodes
are \emph{non-disconnecting} nodes (cf. \cite[\S 1.1]{ck}). Under a semistable deformation $\alpha:\Rc\to\Delta$ of $R$ as in Definition \ref{deformation}, it is  possible to deform such a curve $C$ to a $\delta$-nodal curve on the fibres $R_t$ of 
$\alpha:\Rc\to\Delta$, for $t\neq 0$, 
preserving its marked nodes  and  smoothing the remaining $g+1$ nodes of $C$ located at its intersection with $E$ (see \cite[Lemma\,1.4]{ck}).

Usually we will assume the deformation $\alpha:\Rc\to\Delta$ to be general.
Then, given $C \in V_{\delta}(R)$, we may find a pair $(C',S')$ general in some component $\V \sub \V_{1,\delta}$, where $(S',H')$ is general  in  $\K_p$ and $C' \in V_{1,\delta}(S',H')$, such that $C$  is a flat limit of $C'$. Let $Z$ [resp. $Z'$] be the scheme of the $\delta$ marked nodes of $C$ [resp. of $C'$]. Then  (possibly after shrinking $\Delta$ further) the triple $(R,H,Z)$ is a semistable degeneration of $(S',H',Z')$ and  we may  apply Proposition \ref{thm:main} 
to show the desired vanishings in \eqref{eq:modcond2} needed  to prove Theorem  \ref{thm:mappamod}.
For this we need the 
following result, whose proof we postpone until the next section. 

\begin{prop} \label{lemma:belleintersezioni} 
There exist $W \in |T^1_{R'}|$ and $C \in V_{\delta}(R)$, with $Z$ its scheme of $\delta$ marked nodes, such that:\\
\begin{inparaenum}
\item[(i)] if $g :=p-\delta   \geqslant 15$, then the maps $\sum_{i=1}^ 2 H^0(q_i)$ and $\sum_{i=1}^ 2 H^0(\FF_i)$ are surjective and 
$$h^1(\O_{R_i}(H+2\pi^*_i(F_i)) \* \I_{X_i \cup Y_i \cup Z_i})=h^1(\O_{R_i}(H-2\pi^*_i(F_i)) \* \I_{Z_i})=0, \quad \text{for}\quad i=1,2;$$
\item[(ii)] if $g :=p-\delta  \leqslant7$, then 
$$\im (H^0(\FF_1)) \cap \im (H^0(\FF_2))=\im (H^0(q_1)) \cap \im (H^0(q_2))=\{0\}.$$
\end{inparaenum}
\end{prop}

\bigskip
\begin{proof}[Proof of Theorem \ref{thm:mappamod}] We will prove the desired vanishings in \eqref{eq:modcond2} using Proposition \ref{thm:main}. 

 If  $g \geqslant15$, conditions \eqref{eq:buffon1} and \eqref{eq:buffon2} in Proposition \ref{thm:main} are satisfied by  Proposition \ref{lemma:belleintersezioni}(i), whereas condition \eqref{eq:buffon3} is satisfied by the middle horizontal sequence in \eqref{eq:nculo} and the vanishings of $h^1$ in Proposition \ref{lemma:belleintersezioni}(i). 

 If \color{black} $g \leqslant7$,  conditions \eqref{eq:neuer2} and \eqref{eq:neuer4} in Proposition \ref{thm:main} are satisfied by  Proposition \ref{lemma:belleintersezioni}(ii), whereas condition \eqref{eq:neuer3} is satisfied 
 by the upper horizontal sequence in \eqref{eq:vaffa} and   the fact that  $h^0(\O_{R_i}(H-E\pm 2\pi^*_i(F_i)))=0$, for  $i=1,2$. 
\end{proof}




\section{Proof of Proposition \ref{lemma:belleintersezioni}} \label{sec:lemma}

With a slight abuse of terminology, we will call \emph{lines} the curves on $R_i$ in the pencil $|\pi^ *(F_i)|$, for $i=1,2$. 
The following component of $V_{\delta}(R)$ has been introduced in \cite{cfgk}. 

\begin{definition} \label{def:compopartenza} 
For any $0 \leqslant \delta \leqslant p-1$, we define  $W_{\delta}(R)$ to be the set of  curves $C$ in $V_{\delta}(R)$ such that:\\
\begin{inparaenum} 
\item[(i)] $C$ does not contain any of the exceptional curves $\ee_{i,j}$ of the contractions $\pi_i:R_i \to R'_i$, $i=1,2$; \\
\item[(ii)]  $C$ has exactly $\delta_1:=\lfloor \frac{\delta}{2} \rfloor$ nodes on $R_1-E$ and $\delta_2:=\lceil \frac{\delta}{2} \rceil$ nodes on 
$R_2-E$, hence it splits off $\delta_i$ lines on $R_i$, for $i=1,2$;\\
\item[(iii)] the union of these $\delta=\delta_1+\delta_2$ lines is connected.\\
\end{inparaenum} For any curve $C$ in $W_{\delta}(R)$, we denote by $\mathfrak C$ the connected union of $\delta$ lines as in (iii), called the \emph{line chain of length $\delta$} of $C$, and by $\gamma_i$
the irreducible component of the residual curve to $\mathfrak C$
on $R_i$, for $i=1,2$. 
\end{definition}

It has been proved in \cite[Prop. 4.2]{cfgk} that $W_{\delta}(R)$ is a smooth open subset of a component of $V_{\delta}(R)$, with the nodes described in (ii) as the marked nodes of any of its members.
We write $W_{\delta}(R')$ for the set of images in $R'$ of the curves in $W_{\delta}(R)$. Without further notice, we will denote the image of $C \in W_{\delta}(R)$ and its line chain $\mathfrak C$ and components $\gamma_i$ by $C'$, $\mathfrak C'$ and $\gamma'_i$, respectively.

Members of $W_{\delta}(R)$ are shown in Figure \ref{fig:dis} below.
\begin{figure}[ht] 
\[\begin{array}{ll}
\includegraphics[width=5.2cm]{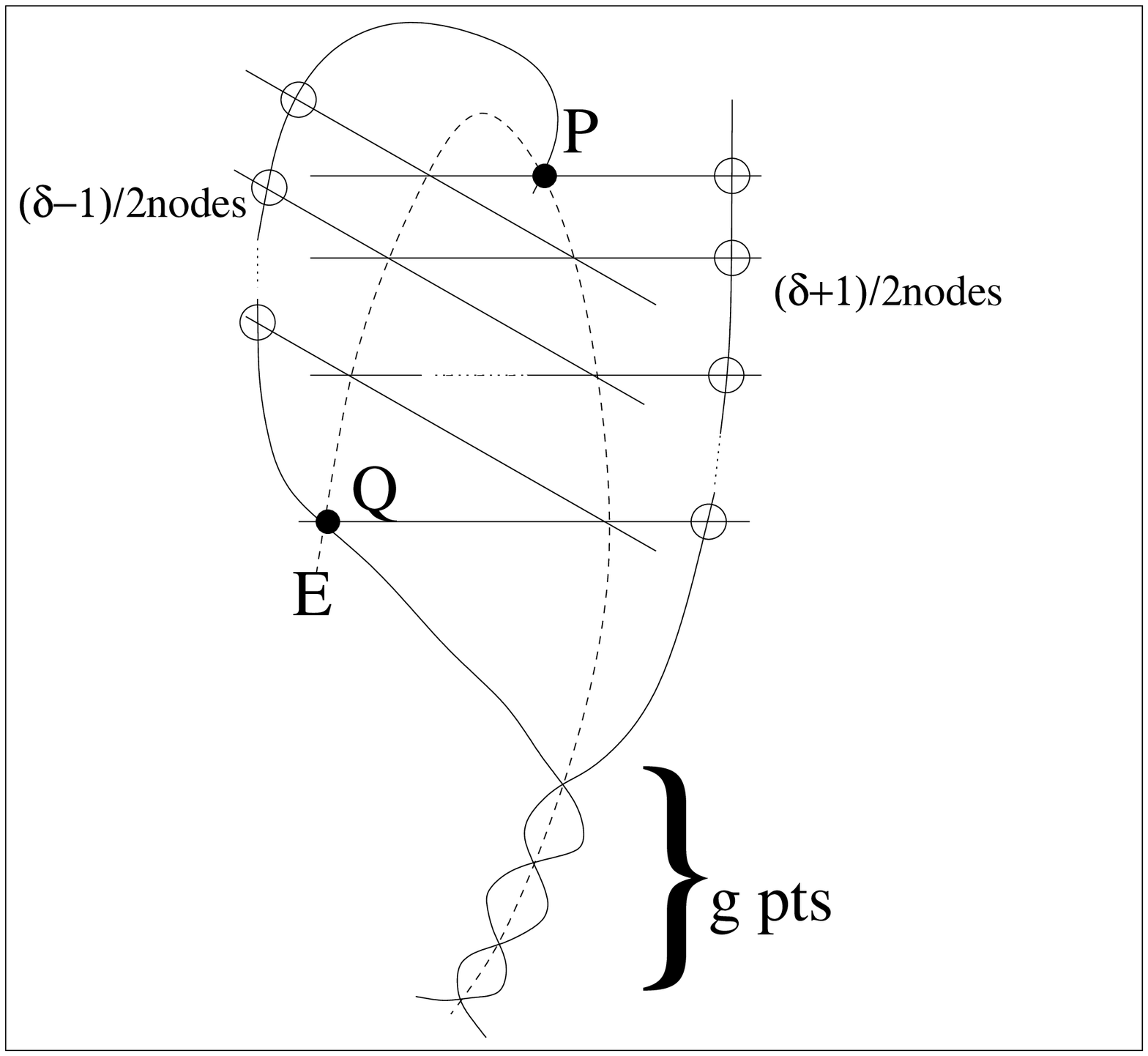} & \includegraphics[width=5cm]{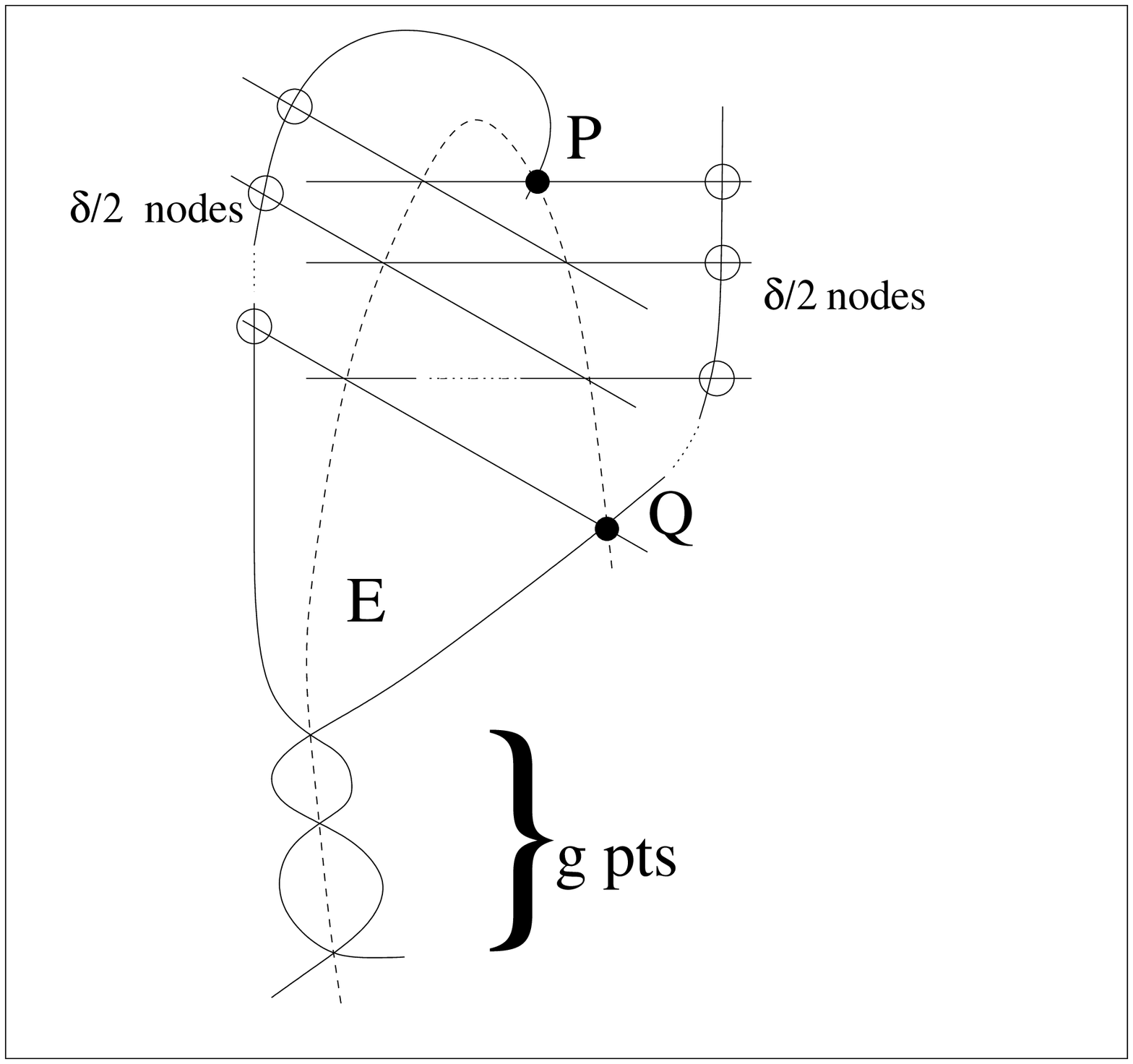} \\
\end{array}
\]
\caption{Members of $W_{\delta}(R)$ when $\delta$ is odd (left) and 
even (right)}
\label{fig:dis}
\end{figure}
The points $P$ and $Q$ in the picture  (the \emph{starting point} and the \emph{end point} of $\mathfrak C$
)
satisfy the following relation on $E$:
\begin{eqnarray}
  \label{eq:relgen-odd}
  P+Q & \sim & \frac{\delta+1}{2}L_2-\frac{\delta-1}{2}L_1, \; \; \mbox{when $\delta$ is odd;} \\
\label{eq:relgen-even}  P-Q &\sim & \frac{\delta}{2}(L_2-L_1), \; \; \mbox{when $\delta$ is even.} 
\end{eqnarray}

We denote by $\mathfrak c$ the intersection of $\mathfrak C$ with $E$, considered as a reduced divisor on $E$. This consists of $\delta+1$ points, i.e., $P+Q$ plus the $\delta-1$ double points of $\mathfrak C$ which are all located on $E$. One has:
\begin{eqnarray}
  \label{eq:relgen-odd_1}
  \mathfrak c & \sim & P+Q+ \frac{\delta-1}{2}L_1 \; \; \mbox{when $\delta$ is odd;} \\
\label{eq:relgen-even_2}  \mathfrak c & \sim & P + \frac{\delta}{2}L_1 \sim  Q + \frac{\delta}{2}L_2 \; \; \mbox{when $\delta$ is even.} 
\end{eqnarray}

Recalling \eqref{eq:defV}, we have:
 
\begin{lemma} \label{lem:V1V2trans}  There exist $W \in |T^1_{R'}|$ and $C \in V_{\delta}(R)$ such that $\dim(V_1 \cap V_2) = \max\{0,g-15\}$.
\end{lemma}

\begin{proof} 
One has
$Y_i, Z_i \not \in E$ whereas $X_i \in E$,  cf. Lemma \ref{lem:ramif}.  By abuse of notation,  we will identify
$Z_i$ and $X_i$ with their images via the maps $\pi_i:R_i \to R'_i$. As indicated before Lemma \ref {lemma:Aaiutino}, one has $W_i = \pi_i(Y_i)$, for $i =1,2$. 

By Leray's spectral sequence and  \eqref{eq:defV}, we  have 
$V_i \cong  H^0(\O_{R'_i}(H'+2F_i) \* \I_{X_i \cup W_i \cup Z_i})_{|E'} \sub H^0(\O_{E'}(H'))$, with the inclusion  explained right after \eqref{eq:defV}: at divisors level,  take a divisor   in 
$|\O_{R'_i}(H'+2 F_i) \* \I_{X_i \cup W_i \cup Z_i}|$, restrict it to $E'$ and then 
remove the points $X_i$, $i =1,2$.  \medskip
 
\noindent 
Case 1:  $p \leqslant 6$.  Since $\deg(W)=16$, there is $h \in \{1,2\}$ such that  $k_h=\deg(W_h) \geqslant 8$. Moreover, 
$\dim (V_{h}) \leqslant h^0(\O_{R'_{h}}(H'+2F_{h}) \* \I_{X_{h} \cup W_{h}})_{|E'}$. From 
\[ 
\xymatrix{
0 \ar[r] &  \O_{R'_{h}}(H'+2 F_{h}-E') \ar[r] &  \O_{R'_{h}}(H'+2 F_{h}) \* \I_{X_{h} \cup W_{h}} \ar[r] &  \O_{E'}(H'+2 F_{h})(-X_{h} -W_{h}) \ar[r] &    0
}
\]and 
$\deg (\O_{E'}(H'+2 F_{h})(-X_{h} -W_{h})) =p+1-k_{h} \leqslant-1$, 
we see that $V_{h} =\{0\}$ and the assertion follows. 

\medskip
 
\noindent 
Case 2:  $p\geqslant 7$.  If we project $ R_i' \subset \PP^p$ from $X_i$, which consists of four points on different fibers of $|F_i|$, we obtain a smooth, rational normal scroll $R''_i \subset \PP^{p-4}$ of the same type  as $R_i'$, where the smoothness of $R_i''$ follows from $p \geqslant7$ and the fact that the projection is an internal projection of the original scroll from the linear span of four points lying on four different fibers.

The linear system  $|\O_{R''_i}(1) |$ on $R_i'' \subset \PP^{p-4}$ is given (with obvious notation) by $|\sigma_i''+(n-2)F''_i|$, which corresponds on $R_i'$ to $|\O_{R'_i}(H') \* \I_{X_i}|$.  Likewise, $|\O_{R'_i}(H'+2F_i) \* \I_{X_i \cup W_i \cup Z_i}|$ on $R_i'$ corresponds  to $|\O_{R''_i}(\sigma_i''+nF''_i) \* \I_{W_i \cup Z_i}|$ on $R_i''$ (by abuse of notation we  denote by the same symbols the images of $W_i$ and $Z_i$ under the projection from $X_i$).

Hence $V_i$ corresponds on $R_i''$ to $H^0(\O_{R''_i}(\sigma_i''+nF''_i) \* \I_{W_i \cup Z_i})_{|E''}$, where $E''\cong E'$ is the image of $E'$ via the projection from the points $X_i$.  
This shows that we can directly work on $R_i'$ and make the identification 
\begin{equation}
  \label{eq:identificaz} V_i = H^0(\O_{R'_i}(H') \* \I_{W_i \cup Z_i})_{|E'} \sub H^0(\O_{E'}(H')).\
\end{equation}Since  $R'_i$ is linearly normal, we have 
$H^0(\O_{R'_i}(H') \* \I_{W_i \cup Z_i})_{|E'}=H^0(\O_{\PP^{p}}(1) \* \I_{(W_i \cup Z_i)/\PP^p})_{|E'}$, hence  
\[ V_1 \cap V_2=H^0(\O_{\PP^{p}}(1) \* \I_{(W \cup Z)/\PP^p})_{|E'}\cong  H^0(\O_{R'}(1) \* \I_{W \cup Z})_{|E'} \cong H^0(\O_{R'}(1) \* \I_{W \cup Z}).
\]

Subcase 2A:  $g\leqslant 14$. Since $15$ of the $16$ points in $W$ are general on $E'$, and since 
$\dim (|\O_{R'}(1) \* \I_{Z}|)=g$  for any $C \in W_{\delta}(R')$ with $Z$ its scheme of $\delta$  marked not disconnecting nodes (cf. \cite[Rem.\;1.1]{ck}),  one has $|\O_{R'}(1) \* \I_{W \cup Z}|= \emptyset$, i.e.,  $V_1 \cap V_2=\{0\}$, as desired. 

Subcase 2B: $p=g=15$. The argument is similar to the one in the previous case.  Indeed, the fact that  $\O_{E'}(W) \simeq T^1_{R'} \not \simeq \O_{E'}(H')$ when $p=15$ (cf. \eqref{eq:B}) implies that $|\O_{R'}(H') \* \I_{W}|$ is empty, hence again $V_1 \cap V_2=\{0\}$. 

Subcase 2C: $g\geqslant 15$ and $p\geqslant 16$. Since $\deg(\O_{E'}(H'-W))=p-15\geqslant 1$, we have
$\dim (|\O_{R'}(H') \* \I_{W}|)=p-16$. 

If $\delta \leqslant5$ (whence $3\delta\leqslant p$), then we can choose a curve $C\in V_{\delta}(R)$ whose 
$\delta$ nodes map to general points on one of the two components of $R'$,
  so that $\dim (|\O_{R'}(1) \* \I_{W \cup Z}|) =\max\{-1,p-16-\delta\}$, as desired.   
 
 Hence we may assume $\delta \geqslant6$, which yields $p =g+\delta \geqslant 21$ and we will construct a curve $C\in W_\delta(R)$ verifying the assertion. 

Let $P\in E'$ be a general point, and consider the line chain $\mathfrak{C}'$ starting at $P$. Set $\Sigma:=
|\O_{\mathbb P^ p}(1)\otimes \mathcal I_{\mathfrak{C}'/\mathbb P^ p}|$. The intersection of $\mathfrak{C}'$ with $E'$ consists of a divisor $\mathfrak c$ of degree $\delta+1$, i.e.,  $P+Q$ plus the $\delta-1$ double points of the line chain $\mathfrak{C}'$.
One has $\dim(\Sigma)=g-1$, and the hyperplanes in $\Sigma$ cut out $E'$ in $\mathfrak c$ plus a divisor $D$ of degree $g$. Hence the linear system $\Sigma$ cuts out on $E'$, off $\mathfrak c$, the complete, base point free, linear system $|D|$ of degree $g$. 

Assume  $g\geqslant 17$.   Take a general hyperplane $H'\in \Sigma$, let $C'=\mathfrak C'+\gamma'_1+\gamma'_2$ be the curve in  $W_{\delta}(R')$ cut out by $H'$ on $R'$ (see Definition \ref {def:compopartenza} and Figure  \ref {fig:dis}), and let $D$ be divisor cut out by $H'$ on $E'$ off $\mathfrak c$. Take $D'$ any effective divisor of degree $15$ contained in $D$. Since $g\geqslant 17$, $D'$ is a general effective divisor of degree $15$ on $E'$.

Let $P'$ be the unique point on $E'$ such that $W:=P'+D'\in  |T^1_{R'}|$.  We claim that 
\begin{equation}
  \label{eq:noncista}
  P'\not\in C', \,\, \text {i.e.}\,\, P' \not \in D+\mathfrak c.
\end{equation}

Indeed if $P'  \in \mathfrak c$,  by an analogue of \eqref {eq:relgen-odd} or \eqref {eq:relgen-even} there would be an integer $k \geqslant1$ such that either
\begin{eqnarray}
  \label{eq:roma1} P+P' & \sim & kL_2-(k-1)L_1, \; \; \mbox{or}\\
\label{eq:roma2}   P-P' & \sim & k(L_2-L_1). 
 \end{eqnarray}Since $T^1_{R'} \simeq \O_{E'}(D'+P')$, then \eqref{eq:B} and \eqref{eq:roma1} combined yield 
 \[
\O_{E'}(P)\simeq \O_{E'}(D')\otimes  \O_{E'}(-4)\otimes L_1^{(p-k-2)} \otimes L_2^{(p+k-3)}. 
   \]
This uniquely determines $P$ once $D', L_1, L_2$ and $\O_{E'}(1)$ have been, as they can, generically chosen, which is a contradiction, since $P$ is also general on $E'$. Similarly by combining  \eqref{eq:B} and \eqref{eq:roma2}.  This proves that $P'  \not\in \mathfrak c$.

To prove that $P'\not \in D$ we note that, by the generality of $D'$, the divisor $P'+D'$ is  general in $|T^1_{R'}|$, hence it is reduced, thus $P'\not \in D'$. Moreover, since $|D-D'|$ is base point free, because $\deg(D-D')\geqslant 2$, we may also assume that $P'\not\in D-D'$. This ends the proof of \eqref {eq:noncista}. 

Next we project $R'_i$ from $D'$. 
By the generality of $D'$, we obtain a  smooth rational normal scroll $R''_i \subset \PP^{p-15}$ for $i=1,2$
(smoothness follows from $p \geqslant21$ and the fact that the projection is an internal projection of the original scroll from the linear span of $D'$), and $R''_1$ and $R''_2$ intersect transversally along a smooth elliptic curve 
$E'' \cong E'$ (the projection of $E'$). Under this projection, the image $C''$ of $C'$ is isomorphic to the curve obtained by normalizing $C'$ at $D'$. 
Thus $C''$ still has $\delta$ marked nodes on the smooth locus of $R''=R''_1\cup R''_2$, whose set we denote by $Z$ as the nodes of $C'$. Since $g\geqslant 17$, the curves $\gamma''_1$ and  $\gamma''_2$, images 
 of $\gamma'_1$ and  $\gamma'_2$, intersect in at least two points on $E''$. Hence, the normalization of $C''$ at the $\delta$ marked nodes is connected. By \cite[Rem.\,1.1]{ck} we have $\dim (|\O_{R''}(1) \* \I_{Z}|)=(p-15)-\delta = g-15$.
Since the hyperplane sections of $R''$ are in one-to-one correspondence with the hyperplane sections of $R'$ passing through $D'$,  this  yields 
\begin{equation*}  \label{eq:dimsu'}
\dim (|\O_{R'}(1) \* \I_{D' \cup Z }|)=g-15.
\end{equation*}

To accomplish the proof we have to exclude that   
\begin{equation}\label{eq:contr}
 |\O_{R'}(1) \* \I_{W \cup Z }|=|\O_{R'}(1) \* \I_{D' \cup Z }|,
 \end{equation}
which would mean that any hyperplane passing through $Z$ and $D'$ contains also $P'$.  If this were  the case, this would in particular happen for the curve $C'$, against \eqref {eq:noncista}. This ends the proof in the case $g\geqslant 17$.

For $g=16$ the proof runs exactly as above. There is only one minor change in the proof of \eqref {eq:noncista}. The proof that  $P'\not\in \mathfrak c+D'$ works with no change.  If $P'\in D-D'$, then $P'= D-D'$, hence $D\in |T^1_{R'}|$. 
Since $\mathfrak c+D\sim \O_{E'}(1)$, then   \eqref{eq:B} yields 
\[
\mathfrak c\sim \O_{E'}(-3)\otimes L_1^{\otimes (p-3)}\otimes  L_2^{\otimes (p-3)}.
\]
By \eqref {eq:relgen-odd_1} and \eqref {eq:relgen-even_2} we have 
\begin{eqnarray*}
 P+Q &\sim  \O_{E'}(-3)\otimes L_1^{\otimes (p-3- \frac{\delta-1}{2})}\otimes  L_2^{\otimes (p-3)} \; \; &\mbox{when $\delta$ is odd;} \\
 P  &\sim \O_{E'}(-3)\otimes L_1^{\otimes (p-3 - \frac{\delta}{2})}\otimes  L_2^{\otimes (p-3)}  \; \; &\mbox{when $\delta$ is even.} 
\end{eqnarray*}
The latter relation contradicts the generality of $P$. The former gives, together with \eqref{eq:relgen-odd_1}, the relation
\[ \O_{E'}(3) \sim L_1^{\*(p-\delta-2)} \* L_2^{\*(p-3-\frac{\delta+1}{2})},\]
contradicting the general choices of $L_1$ and $L_2$.

Let us finally consider the case $g=15$. The basic idea of the proof is the same, so we will be brief. 
We let $D'$ be any effective divisor of degree 14 contained in $D$, so that $D'$ is a general divisor of degree 14 on $E'$. Let $P'$ be the unique point of $E'$ such that
$W:=D'+P+P'\in T^ 1_{R'}$. We claim that \eqref {eq:noncista} still holds. The proof is similar to the ones in the previous cases and can be left to the reader. Then we project $R'$ from $D'$. The projection $C''$ of $C'$ is connected and has $\delta$ marked nodes on the smooth locus of the projection $R''$ of $R'$, whose set we denote by $Z$ as the nodes of $C'$. Then $\dim (|\O_{R''}(1) \* \I_{Z}|)=(p-14)-\delta = 1$, hence
$\dim (|\O_{R'}(1) \* \I_{D' \cup Z }|)=1$.
We claim that 
\begin{equation*} 
|\O_{R'}(1) \* \I_{D' \cup \{P\} \cup Z }|=\{C'\}.
\end{equation*}
Indeed, a curve in $|\O_{R'}(1) \* \I_{D' \cup \{P\} \cup Z }|$ clearly contains the line cycle $\mathfrak C'$, hence it cuts on $E'$ a divisor of degree $p+1$ which contains $\mathfrak c+D'$ whose degree is 
$p$. Hence this curve is uniquely determined. Since $C'\in |\O_{R'}(1) \* \I_{D' \cup \{P\} \cup Z }|$ the assertion follows. Finally, by \eqref {eq:noncista}, we see that $|\O_{R'}(1) \* \I_{W \cup Z }|=\emptyset$, proving the assertion in this case.\end{proof}

\begin{corollary} \label{cor:XWZic} If $g \geqslant15$ then, for $i=1,2$, we have: 
\[ \dim (V_i)=h^0(\O_{R_i}(H+2\pi_i^*(F_i)) \* \I_{X_i \cup W_i \cup Z_i})=p-\delta_i-k_i+1 \; \; \mbox{and} \; \; h^1(\O_{R_i}(H+2\pi_i^*(F_i)) \* \I_{X_i \cup W_i \cup Z_i})=0.\]
\end{corollary}

\begin{proof} By \eqref{eq:cV} and Leray's spectral sequence, we have 
\begin{equation}
\label{eq:dimensione}
\dim (V_i)=h^0(\O_{R'_i}(H'+2F_i) \* \I_{X_i \cup W_i \cup Z_i}) \geqslant p+1-\delta_i-k_i \geqslant 0, \,\, \text{for}\,\, i=1,2,
\end{equation} 
because $p+1-\delta_i-k_i\geqslant p-\delta-16+1 = g-15 \geqslant 0$. 
As $\dim (V_1 \cap V_2)=g-15 \geqslant0$ and $V_1 + V_2 \subseteq H^0(\O_{E'}(1))$, the latter of dimension $p+1$, by the Grassmann formula equality must hold in \eqref {eq:dimensione}. The statement about $h^1$ then  follows. 
\end{proof}

Recalling  \eqref{eq:defU}, we have:

\begin{lemma} \label{lem:U1U2trans} 
 There exist $W \in |T^1_{R'}|$ and $C \in V_{\delta}(R)$, with $Z$ its scheme of $\delta$ marked nodes, such that  $\dim(U_1 \cap U_2) =\max\{0,g-7\}$.
\end{lemma}

\begin{proof} 
It is similar to the one of Lemma \ref{lem:V1V2trans}.  We have 
$$U_i = H^0(\O_{R'_i}(H'-2F_i) \* \I_{Z_i})_{|E'} \subset H^0(\omega_{E'}(H')) = H^0(\O_{E'}(H'));$$ as explained right after \eqref{eq:defU}, at level of divisors the inclusion is given by taking a divisor in  \linebreak 
$|\O_{R_i'}(H' - 2 F_i) \* \I_{Z_i}|$, restricting it to $E'$, then adding the ramification divisor $X_i$, for $i=1,2$. 

As in the proof of Lemma \ref{lem:V1V2trans}, where we showed that we could make the identification \eqref{eq:identificaz}, we can reduce to making the identification 
\[ U_i= H^0(\O_{R'_i}(H') \* \I_{X_i \cup Z_i})_{|E'} \subset H^0(\O_{E'}(H')) \] (recall that $X_i\in |2F_i|$). 
Since  $E'$ and $R'_i$ are linearly normal, we thus have
\[ U_1 \cap U_2=H^0(\O_{\PP^{p}}(1) \* \I_{(X \cup Z)/\PP^p})_{|E'}= H^0(\O_{R'}(1) \* \I_{X \cup Z})_{|E'} \cong H^0(\O_{R'}(1) \* \I_{X \cup Z}).
\]where $X:=X_1+X_2$ as a divisor on $E'$.  We have the exact sequence
\[ \xymatrix{ 
0 \ar[r] & \O_{R'} (1) \* \I_{E'} \ar[r] & \O_{R'} (1) \* \I_{X} \ar[r] & \O_{E'}(H'-X) \ar[r] &  0}.
\]
Since $E'$ is non--degenerate, one has $h^ 0(\O_{R'} (1) \* \I_{E'})=0$. Moreover $\deg( \O_{E'}(H'-X))=p-7$ and, when $p=7$,  one has $\O_{E'}(2X) \simeq L_1^{\otimes 4} \otimes L_2^{\otimes 4} \not \simeq \O_{E'}(2)$. So $h^ 0(\O_{E'}(H'-X))=0$ for $p \leqslant 7$, hence in this case 
$h^ 0(\O_{R'} (1) \* \I_{X})=0$ and we are done. 
We may therefore assume that $p \geqslant8$, in which case $h^ 0(\O_{R'} (1) \* \I_{X})=h^ 0(\O_{E'}(H'-X))=p-7$. 

If  $\delta \leqslant 1$, we may find $C\in V_{\delta}(R)$ whose marked node (if any) maps to a general point on one of the two components of $R'$. Then  it follows that $\dim (|\O_{R'}(1) \* \I_{X \cup Z}|) =\max \{0,p-8-\delta\}$, and we are done. 

We may henceforth assume that $\delta \geqslant 2$ and we will construct a curve $C\in W_\delta(R)$ satisfying the assertion.  Assume $L_1$ and $L_2$ are fixed and {general} on $E'$ as an abstract elliptic curve. Then also $X$ is fixed. We will accomplish the proof by finding a suitable embedding of $E'$ as an elliptic normal curve of degree $p+1$ in $\PP^p$, thus defining $R'_1, R'_2$ and $R':=R'_1 \cup R'_2$ via $|L_1|$ and $|L_2|$, and then finding $C' \in W_{\delta}(R')$ with its scheme $Z$ of  marked nodes such that $U_1 \cap U_2$ 
satisfies the desired condition.

Pick a point $P\in E'$. We assume that $P$ is general if $g\geqslant 8$ whereas we take $P\in X_1$ if $g\leqslant 7$. Consider the line chain $\mathfrak C'$ of length $\delta$ with starting point at $P$ and the divisor $\mathfrak c=P+P_1+\ldots+P_\delta$ cut out by $\mathfrak C'$ on $E'$, so that the lines of $\mathfrak C'$ are spanned by $P+P_1$, $P_1+P_2$, etc. 
We claim that
\begin{eqnarray}
  \label{eq:cl1}
  \mbox{if $g\geqslant 8$ the divisor $\mathfrak c$ is reduced and $\mathfrak c\cap X=\emptyset$;}\\
   \label{eq:cl2}  \mbox{if $g\leqslant 7$ the divisor $\mathfrak c$ is reduced and $\mathfrak c\cap X= P$.}
\end{eqnarray}
We prove only \eqref {eq:cl1}, since the proof of \eqref {eq:cl2} is similar and can be left to the reader. For \eqref {eq:cl1}, by generality, $P$ is not in $X$. Moreover, for each $k\in \{1,\ldots, \delta\}$, one has
\begin{eqnarray}\label{eq:sim-odd}
  P+P_k & \sim & \frac{k+1}{2}L_2-\frac{k-1}{2}L_1  \; \; \mbox{if $k$ is odd},
\\
  \label {eq:sim-even} P-P_k & \sim & \frac{k}{2}(L_2-L_1)  \; \;
 \mbox{if $k$ is even},
 \end{eqnarray}
(see \eqref {eq:relgen-odd} and \eqref {eq:relgen-even}). Hence, if $P_k=P_h$ for $1 \leqslant k<h \leqslant \delta$, then there is a non--trivial relation between $L_1$, $L_2$ and $P$, a contradiction. Similarly, if $2P_k\sim L_i$ for $i=1$ or  $i=2$, then \eqref {eq:sim-odd} and \eqref {eq:sim-even} yield a non--trivial relation between $L_1$, $L_2$ and $P$ a contradiction again.  The same if $P=P_k$ for some $k\in \{1,\ldots, \delta\}$. Hence \eqref{eq:cl1} is proved. 

Fix now an effective subdivisor $\hat{X}$ of $X$ of degree $\min\{7,g-1\}$, not containing $P$ if $g \leqslant 7$, and pick a {general} effective divisor $D$ on $E'$ containing $P$ with
\[\deg (D) =g+1-\deg(\hat X) \geqslant 2.\]
The  divisor
$F:=\hat{X} + D +  P_1 + \cdots + P_{\delta}$
has degree $p+1$ and $|\O_{E'}(F)|$ gives an embedding of $E'$ as an elliptic normal curve of degree $p+1$ in $\PP^p$. The hyperplane cutting out the divisor $F$ on $E'$ cuts out on $R'$ a curve $C' \in W_{\delta}(R')$ with line chain $\mathfrak C'$.
With notation as in Definition \ref{def:compopartenza},  $\gamma'_1 \cap \gamma'_2=\hat X+D-P$.

When we project $R'_i$ from $\hat{X}$ we obtain a smooth rational normal scroll $R''_i$ in $\PP^r$ with $r={\max\{p-7,\delta+1\}}\geqslant 3$ (here we use $\delta \geqslant 2$ and the fact that the projection is an internal projection of a scroll). As above, $R''_1$ and $R''_2$ intersect transversally along a smooth elliptic curve $E'' \cong E'$. Under this projection, the image $C''$ of $C'$ is isomorphic to the curve obtained by normalizing $C'$ at $\hat{X}$. It still has $\delta$ nodes on the smooth locus of $R''$, which we denote by $Z$ as the nodes of $C'$. The images of $\gamma'_1$ and  $\gamma'_2$ under the projection intersect in at least $\deg(D-P)\geqslant 1$ points on $E''$. Hence, the normalization of $C''$ at the nodes lying on the smooth locus of $R''=R''_1\cup R''_2$ is  connected, i.e., the marked nodes are not disconnecting. By \cite[Rem. 1.1]{ck}, we have 
\begin{equation}  \label{eq:dimcazzo}
\dim (|\O_{R''}(1) \* \I_{Z}|)=r-\delta=\max\{g-7,1\}.
\end{equation}
Since the hyperplane sections of $R''$ are in one-to-one correspondence with the hyperplane sections of $R'$ passing through $\hat X$,  \eqref{eq:dimcazzo} yields 
\begin{equation}  \label{eq:dimcazzone}
\dim (|\O_{R'}(1) \* \I_{\hat{X} \cup Z }|)=\max\{g-7,1\}.
\end{equation}

If $g\geqslant 8$, we are done unless 
\begin{equation*}\label{eq:sob}
\dim (|\O_{R'}(1) \* \I_{X \cup Z }|)=\dim (|\O_{R'}(1) \* \I_{\hat{X} \cup Z }|)=g-7
\end{equation*}
which means that if a hyperplane in $\PP^ p$ contains $\hat{X} \cup Z$ it also contains
$X$.  In this case we would have that $X \subset C'$, because $C'\in |\O_{R'}(1) \* \I_{\hat{X} \cup Z }|$, so that
\[ X <  F=\hat{X} + D +  P_1 + \cdots + P_{\delta}=\hat{X} + D + (\mathfrak{c}-P).\]
This  contradicts either \eqref{eq:cl1} or the generality of $D$.

If $g \leqslant 7$, we are done unless $X-\hat X$, which has degree $9-g\geqslant 2$ and contains $P$, imposes 
$c\leqslant 1$ condition to the 1--dimensional system $|\O_{R'}(1) \* \I_{\hat{X} \cup Z }|$. If $c=0$, then, as above, we would have $X<F$ and we find a contradiction. 

Suppose $c=1$.  Then $P$ is not in the base locus $B$ of $|\O_{R'}(1) \* \I_{\hat{X} \cup Z }|$. Otherwise  $\mathfrak C'$ sits in $B$, and  (as at the end of the proof of Lemma \ref {lem:V1V2trans}) $|\O_{R'}(1) \* \I_{\hat{X} \cup Z }|=\{C'\}$, contradicting \eqref {eq:dimcazzone}. 
Thus  $c=1$  yields  
\[|\O_{R'}(1) \* \I_{\hat{X}\cup \{P\} \cup Z }|= |\O_{R'}(1) \* \I_{X \cup Z }|=\{C'\},\]
implying again $X <F$, which leads to a contradiction as above. 

This proves that $c=2$, i.e.,  the assertion.\end{proof}

\begin{corollary} \label{cor:calcoloesp} If $g \geqslant7$ then, for $i=1,2$, we have: 
\[ \dim (U_i)=h^0(\O_{R_i}(H-2\pi_i^*(F_i)) \* \I_{Z_i})=p-3-\delta_i \; \; \mbox{and} \; \; h^1(\O_{R_i}(H-2\pi_i^*(F_i)) \* \I_{Z_i})=0.\]
\end{corollary}

\begin{proof} 
 By \eqref{eq:cU} and Leray, we have $\dim (U_i)=h^0((\O_{R'_i}(H'-2F_i) \* \I_{Z_i}) \geqslant p-3-\delta_i$. 
 As $\dim (U_1 \cap U_2)=g-7 \geqslant 0$, by the Grassmann formula each $U_i$ must have the (expected) dimension. The statement about $h^1$ then  follows. 
\end{proof}

\begin{proof}[Proof of Proposition \ref{lemma:belleintersezioni}] (i) If $g \geqslant 15$, then $h^1(\O_{R_i}(H-2\pi_i^*(F_i)) \* \I_{Z_i})=0$ follows from Corollary \ref{cor:calcoloesp}. Hence, by \eqref{eq:cV}, we have $H^0(\FF_i)=V_i$ and, by Corollary \ref{cor:XWZic},  we have $h^1(\O_{R_i}(H+2\pi_i^*(F_i)) \* \I_{X_i \cup W_i \cup Z_i})=0$. Therefore, the part of the statement Proposition \ref{lemma:belleintersezioni}(i) concerning the vanishing of the $h^1$'s is proved.  The rest  follows by Lemma \ref{lem:V1V2trans}, Corollary \ref{cor:XWZic} and Grassmann formula. 

(ii) If $g \leqslant 7$, the statement  follows by  Lemmas \ref{lem:U1U2trans}, \ref{lemma:VU} and \ref{lem:V1V2trans}.
\end{proof}

%
%

\end{document}